\documentclass[12pt,reqno]{amsart}
\usepackage{geometry}
\usepackage[latin1]{inputenc}
\usepackage[italian,english]{babel}
\usepackage{amsmath, amsfonts, amsthm}
\usepackage{paralist}
\numberwithin{equation}{section}
\newtheorem{thm}{Theorem}[section]
\newtheorem{lem}[thm]{Lemma}

\newtheorem{prop}[thm]{Proposition}

\newtheorem{defn}[thm]{Definition}
\theoremstyle{definition}
\newtheorem{rem}[thm]{Remark}
\theoremstyle{remark}

\newcommand{\ds}{\displaystyle}

\newcommand{\de}{\partial}

\newenvironment{sistema}%
{\left\{\begin{array}{@{}l@{}}}{\end{array}\right.}
\patchcmd{\abstract}{\scshape\abstractname}{\textbf{\abstractname}}{}{}
\makeatletter 
\def\@makefnmark{} 
\makeatother 

\title{Sharp estimates for the first $p$-Laplacian eigenvalue and for the $p$-torsional rigidity  on convex sets with  holes}
\author[G. Paoli, G. Piscitelli,  L. Trani]{
	Gloria Paoli, Gianpaolo Piscitelli, Leonardo Trani}
\address{Dipartimento di Matematica e Applicazioni ``R. Caccioppoli'', Universit\`a degli studi di Napoli Federico II \\ Via Cintia, Complesso Universitario Monte S. Angelo, 80126 Napoli, Italy.}
\email{gloria.paoli@unina.it}
\address{Dipartimento di Ingegneria Elettrica e dell'Informazione \lq\lq M. Scarano\rq\rq, Universit\`a degli Studi di Cassino e del Lazio Meridionale\\ Via G. Di Biasio n. 43, 03043 Cassino (FR), Italy.}
\email{gianpaolo.piscitelli@unicas.it}
\address{Dipartimento di Matematica e Applicazioni ``R. Caccioppoli'', Universit\`a degli studi di Napoli Federico II \\ Via Cintia, Complesso Universitario Monte S. Angelo, 80126 Napoli, Italy.}
\email{leonardo.trani@unina.it}

\begin{document}
\maketitle
\markright{SHARP ESTIMATES FOR THE FIRST $p$-LAPLACIAN EIGENVALUE}
\begin{abstract}
We study, in dimension $n\geq2$, the eigenvalue problem and the torsional rigidity for the $p$-Laplacian on convex sets with holes, with external Robin boundary conditions and internal Neumann boundary conditions. We prove that the annulus maximizes  the first eigenvalue and minimizes the torsional rigidity when  the measure and the external perimeter are fixed.
\\

\noindent MSC 2010: 35J25 - 35J92 - 35P15 - 47J30 \\
\noindent Keywords: Nonlinear eigenvalue problems - torsional rigidity - mixed boundary conditions - optimal estimates
\end{abstract}

\section{Introduction}
In this paper we study the  $p$-Laplacian  operator
\begin{equation*}
	-\Delta_p u:=-{\rm div}\left(|D u|^{p-2}D u\right)
	\end{equation*} 
	 defined on a convex set   $\Omega$ of $\mathbb R^n$, $n\geq2$, that contains  holes; more precisely we are considering  sets $\Omega$ of the form $\Omega =\Omega_0\setminus\overline \Theta$, where $\Omega_0\subseteq\mathbb{R}^n$ is an open bounded and convex set  and $\Theta\subset\subset \Omega_0$ is a finite union of sets, each of one  homeomorphic to a ball.
	In this setting, we study the eigenvalue problem and the torsion problem for the $p$-Laplacian  operator  and the boundary conditions that we impose are of Robin type on the exterior boundary $\Gamma_0:=\de \Omega_0$ and of Neumann type on the interior boundary $\Gamma_1:=\de \Theta$.
The first quantity we deal with is
	\begin{equation}\label{eigRNintro}
		\lambda_p^{RN}(\beta, \Omega)=  \min_{\substack{w\in W^{1,p}(\Omega)\\w\not \equiv0}} \dfrac{\ds\int _{\Omega}|Dw|^p\;dx+\beta\ds\int_{\Gamma_0}|w|^p \;d\mathcal{H}^{n-1}}{\ds\int_{\Omega}|w|^p\;dx}.
	\end{equation}	
This minimization problem variationally characterize the first eigenvalue, i.e. the lowest eigenvalue, of the following:
	\begin{equation}\label{case1intro}
		\begin{cases}
			-\Delta_p u=\lambda_p^{RN}(\beta, \Omega) |u|^{p-2} u & \mbox{in}\ \Omega\vspace{0.2cm}\\
			|D u|^{p-2}\dfrac{\de u}{\de \nu}+\beta |u|^{p-2}u=0&\mbox{on}\ \Gamma_0\vspace{0.2cm}\\ 
			|Du|^{p-2}\dfrac{\de u}{\de \nu}=0&\mbox{on}\ \Gamma_1,
		\end{cases}
	\end{equation}
	where $\de u/\de \nu$ is the outer normal derivative of $u$ and $\beta\in\mathbb{R}\setminus \{0 \}$ is the boundary parameter. 
By the way, we will only consider non zero values of the boundary parameter $\beta$, since the case $\beta=0$ is trivial, being the first eigenvalue identically zero and the relative eigenfunctions constant.
 
In our  first result  (Theorem \ref{main_1}), we prove that among the domains $\Omega$  defined as above, the annulus $A_{r_1,r_2}=B_{r_2}\setminus \overline{B_{r_1}}$, having the same measure of $\Omega$ and such that $P(B_{r_2})=P(\Omega_0)$, maximizes the first $p$-Laplacian eigenvalue, i.e
\begin{equation}
\label{faber}
\lambda_p^{RN}(\beta, \Omega)\leq \lambda_p^{RN}(\beta, A_{r_1,r_2}).
\end{equation}
Inequality \eqref{faber} is a Faber-Krahn (\cite{F,K}) type inequality. In particular, when $\beta\to+\infty$, our Theorem \ref{main_1} gives an answer to the open problem $5$ in \cite[Chap. 3]{Hen}, restricted to convex sets with holes. 

 If we consider the case of $\Omega $ bounded domain of $\mathbb{R}^n$ with Dirichlet boundary conditions on the whole $\de\Omega$, the classical Faber-Krahn inequality says that the first eigenvalue is minimized by the ball among sets of the same volume (see e.g. \cite{D2} and the references therein); for optimization of eigenvalues we also refer, for example, to \cite{AB, P, PS} .

If $\Omega\subseteq\mathbb{R}^n$ is a bounded set with Robin boundary condition on $\de\Omega$ and  $\beta>0$, it has been proved in \cite{B,D} for $p=2$ and in \cite{BD,DF} for $p\in]1,+\infty[$, that the first eigenvalue is minimized by the ball among domains of the same volume. 
On the other hand if $\beta$ is negative,  the problem is still open; in this direction, in \cite{FK} the authors showed that  the ball is a maximizer in the plane only  for small value of the parameter, having fixed the volume. Moreover,  if  $\beta$ is negative and  we fix the perimeter rather than the volume,  the ball maximizes the first eigenvalue among all  open, bounded, convex, smooth enough sets   (see \cite{AFK, BFNT}). Furthermore, we remark that in the case of a general Finsler metric, similar results holds for the anisotropic $p$-Laplacian with Dirichlet (\cite{BFK,DGP1}), Neumann (\cite{DGP2,Pi}) or Robin (\cite{GT,PT}) boundary conditions.

Makai \cite{M} and P\'olya \cite{Po} introduced the method of interior parallels, used by Payne and Weinberger in \cite{PW}, to study the Laplacian eigenvalue problem with external Robin boundary condition and with Neumann internal boundary condition in the plane. 
In our paper, we generalize these tools in any dimension. 
More precisely, our proof is based  on the use of the web functions, particular test functions used e.g. in \cite{ BNT, BFNT, CFG}, and on the study of their level sets. 

Similarly, but only for positive value of $\beta$, we also study the $p$-torsional rigidity type problem:
\begin{equation*}
\dfrac{1}{T_p^{RN}(\beta, \Omega)}= \min_{\substack{w\in W^{1,p}(\Omega)\\ w\not \equiv0}} \dfrac{\ds\int _{\Omega}|Dw|^p\;dx+\beta\ds\int_{\Gamma_0}|w|^p \;d\mathcal{H}^{n-1}}{\left|\ds\int_{\Omega}w\;dx\right|^p};
\end{equation*}
in particular, this problem leads to, up to a suitable normalization, 
\begin{equation*}
\begin{cases}
-\Delta_p u= 1 & \mbox{in}\ \Omega\vspace{0.2cm}\\
|D u|^{p-2}\dfrac{\de u}{\de \nu}+\beta|u|^{p-2} u=0&\mbox{on}\ \Gamma_0\vspace{0.2cm}\\
|Du|^{p-2}\dfrac{\de u}{\de \nu}=0&\mbox{on}\ \Gamma_1.
\end{cases}
\end{equation*}
The second main result of this paper (Theorem \ref{main_2}) states  that, among the domains $\Omega$  defined as in the beginning,  the annulus $A_{r_1,r_2}=B_{r_2}\setminus \overline{B_{r_1}}$, having the same measure of $\Omega$ and such that $P(B_{r_2})=P(\Omega_0)$, minimizes the $p$-torsional rigidity, i.e. 
  \begin{equation}
  \label{saint}
  T_p^{RN}(\beta, \Omega)\geq T_p^{RN}(\beta, A_{r_1,r_2}).
 \end{equation}


The equation \eqref{saint}, when $\Theta=\emptyset$, $p=2$ and Dirichlet boundary condition holds on the whole boundary, is the Saint-Venant inequality, by the name of the authors that first conjectured that the ball in the plane (under area constraint) gives the maximum in quantity \eqref{saint}. This is a relevant problem in the elasticity theory of beams \cite[Sec.35]{So}. It is known that the ball maximizes the torsional rigidity with Robin boundary conditions \cite{BG} among bounded open sets with Lipschitz boundary and given measure. Related results for the spectral optimization problems involving the rigidity are obtained also, for example, in \cite{BB, BBV, Br}.

Finally, we recall that the eigenvalue problem in the plane with reversed boundary conditions, i.e. Neumann on the external boundary $\de \Omega_0$ and Robin on the internal boundary $\de\Theta$,  has been studied in \cite{Her}. Moreover in \cite{DP} the authors generalize this result in dimension $n\geq 2$ for the eigenvalue problem and for the torsional rigidity.

The paper is organized as follows. In the Section 2 we introduce some notations and preliminaries and in the Section 3 we prove the main results.

\section{Notation and preliminaries}
\noindent

In the following,  by $|\Omega|$ we denote the $n-$dimensional Lebesgue measure of $\Omega$, by $P(\Omega)$ the perimeter of $\Omega$, by $\mathcal{H}^k$  the $k-$dimensional Hausdorff measure in $\mathbb{R}^n$. The unit open  ball in $\mathbb{R}^n$ will be denoted by $B_1$ and  $\omega_n:=|B_1|$. 
More generally, we denote with $B_r(x_0)$ the set $x_0+rB_1$, that is the ball centered at $x_0$ with measure $\omega_nr^n$, and by $A_{r_1,r_2}$ the open annulus $B_{r_2}\setminus \overline{B}_{r_1}$, where $\overline{B}_{r_1}$ is the closed ball centered at the origin and $r_1<r_2$.

Throughout this paper, we denote by $\Omega$ a set such that $\Omega =\Omega_0\setminus\overline \Theta$, where $\Omega_0\subseteq\mathbb{R}^n$ is an open bounded and convex set  and $\Theta\subset\subset \Omega_0$ is a finite union of sets, each of one  homeomorphic to a ball of $\mathbb{R}^n$ and with Lipschitz boundary.
We define $\Gamma_0:=\de \Omega_0$ and  $\Gamma_1:=\de \Theta$.


\subsection{Eigenvalue problems}
Let $1<p<+\infty$, 
we deal with the following $p$-Laplacian eigenvalue problem:
\begin{equation}\label{case1}
\begin{cases}
	-\Delta_p u=\lambda_p^{RN}(\beta, \Omega) |u|^{p-2} u & \mbox{in}\ \Omega\vspace{0.2cm}\\
|Du|^{p-2}\dfrac{\de u}{\de \nu}+\beta |u|^{p-2}u=0&\mbox{on}\ \Gamma_0\vspace{0.2cm}\\ 
|Du|^{p-2}\dfrac{\de u}{\de \nu}=0&\mbox{on}\ \Gamma_1.
\end{cases}
\end{equation}
We denote by $\de u/\de \nu $ the outer normal derivative of $u$ on the 
boundary and by  $\beta\in\mathbb{R}\setminus\{ 0\}$ the Robin boundary parameter,  observing that the case $\beta =+\infty$ gives asimptotically the Dirichlet boundary condition. Now we give the definition of eigenvalue and eigenfunction of problem \eqref{case1}.
\begin{defn}
	The real number $\lambda$ is an eigenvalue of \eqref{case1}   if and only if there exists a function $u\in W^{1,p}(\Omega)$, not identically zero, such that 
	\begin{equation*}
		\int_{\Omega}|D u|^{p-2}D uD\varphi \;dx+\beta \int_{\Gamma_0} |u|^{p-2}u\varphi\;d\mathcal{H}^{n-1}=\lambda\int_{\Omega}|u|^{p-2}u \varphi \;dx
	\end{equation*}
	for every $\varphi\in W^{1,p}(\Omega)$. The function $u$ is called  eigenfunction associated to $\lambda$.
\end{defn}
In order to compute the first eigenvalue we use the variational characterization, that is
\begin{equation}\label{eigRN}
\lambda_p^{RN}(\beta, \Omega)=\min_{\substack{w\in W^{1,p}(\Omega)\\ w\not \equiv0}}J_0[\beta,w]
\end{equation}
where 
$$ J_0[\beta,w]:=\dfrac{\ds\int _{\Omega}|Dw|^p\;dx+\beta\ds\int_{\Gamma_0}|w|^p \;d\mathcal{H}^{n-1}}{\ds\int_{\Omega}|w|^p\;dx}$$ 
We observe that $\Omega_0$ is convex and hence it has Lipschitz boundary; this ensures  the existence of  minimizers of the analyzed problems. 
\begin{prop}
Let $\beta\in\mathbb{R}\setminus \{0\}$. There exists a minimizer $u\in W^{1,p}(\Omega)$ of \eqref{eigRN},  which is a weak solution to \eqref{case1}. 
\end{prop}
\begin{proof} First we  consider the case $\beta>0$. Let $u_k\in W^{1,p}(\Omega)$ be a minimizing sequence of \eqref{eigRN} such that $||u_k||_{L^p(\Omega)}=1$. 
Then, being $u_k$ bounded in $W^{1,p}(\Omega)$, there exist a subsequence, still denoted by $u_k$, and a function $u\in W^{1,p}(\Omega)$ with $||u||_{L^p(\Omega)}=1$, such that $u_k\to u$ strongly in $L^p(\Omega)$ and almost everywhere and $D u_k\rightharpoonup D u$ weakly in $L^p(\Omega)$. As a consequence,
\phantom{ }$u_k$ converges strongly  to $u$ in $L^p(\de \Omega)$ and so almost everywhere on $\de \Omega$ to  $u$. Then, by weak lower semicontinuity:
\begin{equation*}
	\lim\limits_{k\to+\infty }J_0[\beta, u_k]\geq J_0[\beta, u].
\end{equation*}

We consider now the case $\beta<0$.  Let $u_k\in W^{1,p}(\Omega)$ be a minimizing sequence of \eqref{eigRN} such that $||u_k||_{L^p(\Omega)}=1$.  Now, since $\beta$ is negative, we have the equi-boundness of the functional $J_0[\beta,\cdot]$, i.e. there exists a constant $C<0$ such that $J_0[\beta,u_k]\leq C$ for every $ k\in \mathbb{N}$.
As a consequence
\begin{equation*}
||Du_k||^p_{L^p(\Omega)}-C||u_k||^p_{L^p(\Omega)}\leq -\beta,
\end{equation*}
and so
\begin{equation*}
	||u||^p_{W^{1,p}(\Omega)}\leq L,
\end{equation*}
where $L:=-\beta/\min\{1,-C \}$.
Then,  there exist a subsequence, still denoted by $u_k$, and a function $u\in W^{1,p}(\Omega)$ such that $u_k\to u$ strongly in $L^p(\Omega)$ and $D u_k\rightharpoonup D u$ weakly in $L^p(\Omega)$. So $u_k$ converges strongly to $u$ in $L^p(\de\Omega)$, 
and so
\begin{equation*}
J_0[\beta,u]\leq \liminf\limits_{k\to \infty }J_0[\beta,u_k]=\inf_{\substack{v\in W^{1,p}(\Omega)\\ v\not \equiv0}} J_0[\beta,v].
\end{equation*}
Finally, $u$ is strictly positive in $\Omega$ by the Harnack inequality (see \cite{T}).

\end{proof} 
Now we  state  some basic properties on the sign and the monotonicity of the first eigenvalue.
\begin{prop}\label{sign} 
	If $\beta>0$, then $\lambda_p^{RN}(\beta,\Omega)$  is positive and   if $\beta <0$, then $\lambda_p^{RN}(\beta,\Omega)$  is negative. Moreover, for all $\beta\in\mathbb{R}\setminus \{ 0\}$, $\lambda_p^{RN}(\beta,\Omega)$ is simple, that is all the associated eigenfunctions are scalar multiple of each other and can be taken to be positive.
\end{prop}
\begin{proof}
Let $\beta>0$, then trivially  $\lambda_p^{RN}(\Omega)\geq0$. We prove that $\lambda_p^{RN}(\Omega)>0$  by contradiction, assuming that $\lambda_p^{RN}(\Omega)=0$. Thus, we consider a non-negative minimizer $u$ such that  $||{u}||_{L^p(\Omega)}=1$ and
\begin{equation*}
0=\lambda_p^{RN}(\Omega,\beta)=\ds\int_{\Omega}|D{u}|^p\;dx+\beta\int_{\Gamma_0}|{u}|^p\;d\mathcal{H}^{n-1}.
\end{equation*}
So, ${u}$ has to be constant in $\Omega$ and consequently ${u}$ is $0$ in $\Omega$, which contradicts the fact that the norm of ${u}$ is unitary.

If  $\beta<0$,  choosing the constant as test function in \eqref{eigRN}, we obtain
\begin{equation*}
\lambda_p^{RN}(\beta, \Omega)\leq\beta\dfrac{P(\Omega_0)}{|\Omega|}<0.
\end{equation*}
Let $u\in W^{1,p}(\Omega)$ be a function that achieves the infimum in \eqref{eigRN}.  First of all we observe that 
$$J_0[\beta, u]=J_0[\beta,|u|] ,$$
and this fact implies that any eigenfunction must have constant sign on $\Omega$ and so we can assume that $u\geq 0$.
In order to prove the simplicity of the eigenvalue, we proceed as in \cite{BK,DG}. We give here a sketch of the proof. Let $u,w$ be positive minimizers of the functional $J_0[\beta, \cdot]$, such that $||u||_{L^p(\Omega)}=||w||_{L^p(\Omega)}=1$. We define $\eta_t=\left( t u^p+(1-t)w^p  \right)^{1/p}$, with $t\in[0,1]$ and we have that $||\eta_t||_{L^p(\Omega)}=1$.  It holds that 
\begin{equation}\label{equality}
J_0[\beta,u]=\lambda_p^{RN}(\beta,\Omega)=J_0[\beta,w].
\end{equation}
Moreover by convexity the following inequality holds true:
\begin{equation}\label{eqconc}\begin{split}
|D\eta_t|^p & =\eta_t^p\left|\frac{t u^p \frac{D u}u+(1-t)w^p\frac{D w}{w}}{tu^p+(1-t)w^p}\right|^p\\
 & \leq \eta_t^p\left[\frac{t u^p}{tu^p+(1-t)w^p}\left|\frac{Du}u\right|^p+\frac{(1-t)w^p}{tu^p+(1-t)w^p}\left|\frac{Dw}w\right|^p\right]=t |Du|^p+(1-t)|Dw|^p.
\end{split}
\end{equation}
Using now \eqref{equality}, we obtain
\begin{equation*}
\lambda_p^{RN}(\beta,\Omega)\leq J_0[\beta,\eta_t]\leq tJ_0[\beta, u]+(1-t)J_0[\beta,w]=\lambda_p^{RN}(\beta,\Omega),
\end{equation*}
and then $\eta_t$ is a minimizer for $J_0[\beta,\cdot]$. So inequality \eqref{eqconc} holds as equality, and therefore $\frac{Du}u=\frac{Dw}w$. This implies that $D(\log u-\log w)=0$, that is $\log\frac u w=const$. We conclude passing to the exponentials.
\end{proof}

	\begin{prop} \label{monot} 
		The map $\beta\to \lambda_p^{RN}(\beta,\Omega)$ is Lipschitz continuous and non-decreasing with respect to $\beta\in\mathbb{R}$.
	Moreover $\lambda_p^{RN}(\beta,\Omega)$ is concave in $\beta$.
	\end{prop}
	
\begin{proof}

Let $\beta_1,\beta_2\in\mathbb{R}$ such that $\beta_1<\beta_2$ and let   $ w\in W^{1,p}(\Omega)$ be  not identically $0$. We observe that
\begin{equation*}
\ds\int_{\Omega}|D{w}|^p\;dx+\beta_1\int_{\Gamma_0}|{w}|^p\;d\mathcal{H}^{n-1}\leq \ds\int_{\Omega}|D{w}|^p\;dx+\beta_2\int_{\Gamma_0}|{w}|^p\;d\mathcal{H}^{n-1}.
\end{equation*}
	Now, passing to the infimum on $w$ and taking into account the variatiational characterization, we obtain
	$\lambda_p^{RN}(\beta_1,\Omega)\leq\lambda_p^{RN}(\beta_2,\Omega)$.
	
	We prove that $\lambda_p^{RN}(\beta,\Omega)$ is concave in $\beta$. Indeed, for fixed $\beta_0\in\mathbb{R}$, we have to show that
	\begin{equation}\label{concavity}
		\lambda_p^{RN}(\beta,\Omega)\leq\lambda_p^{RN}(\beta_0,\Omega)+\left(\lambda_p^{RN}\right)'(\beta_0,\Omega)\left(\beta-\beta_0\right),
	\end{equation}
	for every $\beta\in\mathbb{R}$. 
	Let $w_0$ be the eigenfunction associated to $\lambda_p^{RN}(\beta_0,\Omega)$ and normalized such that  $\int_{\Omega}w_0^p\;dx=1$.
Hence, we have 
	\begin{equation}\label{c}
		\lambda_p^{RN}(\beta,\Omega)\leq \ds\int_{\Omega}|D{w_0}|^p\;dx+\beta\int_{\Gamma_0}|{w_0}|^p\;d\mathcal{H}^{n-1}.
	\end{equation}
	Now, summing and subtracting to the right hand side of \eqref{c} the quantity \\ \mbox{$\beta_0\int_{\Gamma_0}|w_0|^pd\mathcal{H}^{n-1}$}, taking into account that 
	$$ \lambda_p^{RN}(\beta_0,\Omega)=\ds\int_{\Omega}|D{w_0}|^p\;dx+\beta_0\int_{\Gamma_0}|{w_0}|^p\;d\mathcal{H}^{n-1},$$
	and the fact that $$ \left(\lambda_p^{RN}\right)'(\beta_0,\Omega)=\int_{\Gamma_0}|{w_0}|^p\;d\mathcal{H}^{n-1},$$ we obtain the desired result \eqref{concavity}. 
\end{proof}

Now we state a  result  relative to  the  eigenfunctions of problem \eqref{case1} on the annulus.

\begin{prop} \label{radial_theorem}
Let $r_1,r_2$ be two nonnegative real number such that $r_2> r_1$, and let $u$ be the minimizer of problem \eqref{eigRN}  on the annulus $A_{r_1,r_2}$. Then $u$ is strictly positive and radially symmetric, in the sense that $u(x)=:\psi(|x|)$. Moreover, if $\beta>0$, then $\psi'(r)<0$  and if $\beta <0$, then $\psi'(r)>0$.
\end{prop}

\begin{proof}
	The first claim follows from the simplicity of $\lambda_p^{RN}(\beta, A_{r_1,r_2})$ and from the rotational invariance of problem \eqref{case1}. For the second claim, we consider the problem \eqref{case1} with the boundary parameter $\beta>0$.  
The associated radial problem is:
\begin{equation*}
\begin{cases}
-\dfrac{1}{r^{n-1}}\left(|\psi'(r)|^{p-2} \psi'(r)r^{n-1}\right)'=\lambda_p^{RN}(\beta, A_{r_1,r_2}) \psi^{p-1}(r)\quad \text{if}\  r\in(r_1,r_2),\vspace{0.2cm}\\
\psi'(r_1)|\psi'(r_1)|^{p-2} =0, \vspace{0.2cm}\\
|\psi'(r_2)|^{p-2}\psi'(r_2)+\beta\psi^{p-1}(r_2)=0.
\end{cases}
\end{equation*}
We observe that for every $r\in(r_1,r_2)$
\begin{equation}\label{sign_derivative}
-\dfrac{1}{r^{n-1}}\left(|\psi'(r)|^{p-2} \psi'(r)r^{n-1}\right)'=\lambda^{RN}_p(\beta, A_{r_1,r_2}) \psi^{p-1}(r)>0,
\end{equation}
and, as a consequence,
\begin{equation*}
\left(|\psi'(r)|^{p-2} \psi'(r)r^{n-1}\right)'<0.
\end{equation*}
Taking into account the boundary conditions $\psi'(r_1)=0$, it follows that $\psi'(r)<0$, since
\begin{equation*}
|\psi'(r)|^{p-2}\psi'(r)r^{n-1}<0.
\end{equation*}

If $\beta<0$, by Remark \ref{sign}, $\lambda_p^{RN}(\beta, A_{r_1,r_2})<0$ and consequently the left side of the equation \eqref{sign_derivative} is negative, and hence 
 $\psi'(r)>0$.

\end{proof}

\subsection{Torsional rigidity}
Let $\beta>0$, we consider the torsional rigidity for the $p-$Laplacian. More precisely, we are interested in 
\begin{equation}\label{torRN}
\dfrac{1}{T_p^{RN}(\beta, \Omega)}=\min_{\substack{w\in W^{1,p}(\Omega)\\ w\not \equiv0}}K_0[\beta,w],
\end{equation}
where 
$$K_0[\beta,w]: =\dfrac{\ds\int _{\Omega}|Dw|^p\;dx+\beta\ds\int_{\Gamma_0}|w|^p \;d\mathcal{H}^{n-1}}{\left|\ds\int_{\Omega}w\;dx\right|^p}.$$
Problem \eqref{torRN}, up to a suitable normalization, leads to 
\begin{equation}\label{case1TRN}
\begin{cases}
-\Delta_p u=1 & \mbox{in}\ \Omega\vspace{0.2cm}\\
|Du|^{p-2}\dfrac{\de u}{\de \nu}+\beta|u|^{p-2} u=0&\mbox{on}\ \Gamma_0\vspace{0.2cm}\\
|Du|^{p-2}\dfrac{\de u}{\de \nu}=0&\mbox{on}\ \Gamma_1.
\end{cases}
\end{equation}
In the following, we state some results for the torsional rigidity, analogously to the ones stated in the previous section for the eigenvalue problems. The proofs can be easily adapted.

\begin{prop} Let $\beta >0$, then the following properties hold. 
\begin{itemize}
\item  	
There exists a positive minimizer $u\in W^{1,p}(\Omega)$ of \eqref{torRN} which is a weak solution to \eqref{case1TRN}  in $\Omega$. 

\item 	Let $r_1,r_2$ be two nonnegative real numbers such that $r_2> r_1$, and $\psi$ be the minimizer of \eqref{torRN}  on the annulus $A_{r_1,r_2}$. Then $\psi$ is strictly positive, radially symmetric and strictly decreasing.
\item The map $\beta\mapsto \dfrac 1{T_p^{RN}(\beta,\Omega)}$  is positive, Lipschitz continuous, non-increasing and concave with respect to $\beta$.
\end{itemize}
\end{prop}

\subsection{Quermassintegrals: definition and some properties}
For the content of this section we refer, for instance, to \cite{Sc}.
Let $\emptyset\neq \Omega_0\subseteq\mathbb{R}^n$ be  an open, compact and convex set. 
We define the outer  parallel body of $\Omega_0$ at distance $\rho$ as the  Minkowski sum
$$ \Omega_0+\rho B_1=\{ x+\rho y\in\mathbb{R}^n\;|\; x\in\Omega_0,\;y\in B_1 \}.$$
The Steiner formula asserts that
\begin{equation}\label{general_steiner}
|\Omega_0+\rho B_1|=\sum_{i=0}^{n}\binom{n}{i} W_i(\Omega_0)\rho^i. 
\end{equation}
and
\begin{equation}\label{general_steiner}
P(\Omega_0+\rho B_1)=n\sum_{i=0}^{n-1}\binom{n}{i} W_{i+1}(\Omega_0)\rho^i. 
\end{equation}
The coefficients $W_i(\Omega_0)$ are known as quermassintegrals and some of them  have an easy interpretation:
\begin{equation}
W_0(\Omega_0)=|\Omega_0|;\qquad n W_1(\Omega_0)=P(\Omega_0);\qquad W_n(\Omega_0)=\omega_n.
\end{equation}
Furthermore, we  have that
\begin{equation}\label{second W}
\lim\limits_{\rho \to 0^+} \dfrac{P(\Omega_0+\rho B_1)-P(\Omega_0)}{\rho}=n(n-1) W_2(\Omega_0)
\end{equation}
We  recall also  the Aleksandrov-Fenchel inequalities
\begin{equation}\label{aleksandrov-fenchel}
	\left(\dfrac{W_j(\Omega_0)}{\omega_n}\right)^{\frac{1}{n-j}}\geq \left(\dfrac{W_i(\Omega_0)}{\omega_n}\right)^{\frac{1}{n-i}},
\end{equation}
for $0\leq i<j<n$, with equality if and only if $\Omega_0$ is a ball.  If we put in the last inequality $i=0$ and $j=1$, we obtain the classical isoperimetric inequality, that is:
\begin{equation*}
P(\Omega_0)^{\frac{n}{n-1}}\geq n^{\frac{n}{n-1}} \omega_n^{\frac{1}{n-1}} |\Omega_0|.
\end{equation*}
We will also need the case in \eqref{aleksandrov-fenchel} when $i=1$ and $j=2$:
\begin{equation}\label{iso_W_2}
W_2(\Omega_0)\geq n^{-\frac{n-2}{n-1}}\omega_n^{\frac{1}{n-1}}P(\Omega_0)^{\frac{n-2}{n-1}}.
\end{equation}
In the next sections, we will denote by $d_e(x)$ the distance function  from the  boundary of $\Omega_0$. We
use the following notations:
\begin{equation*}
\Omega_{0,t}=\{ x\in\Omega_0\;:\; d_e(x)>t  \},\qquad t\in[0,r_{\Omega_0}],
\end{equation*}
where by $r_{\Omega_0}$ we denote the inradius of $\Omega_0$.
We state now the following two lemmas, whose proofs can be found in \cite{BNT} and \cite{BFNT}.
\begin{lem} 
	Let $\Omega_0$ be a bounded, convex, open set in $\mathbb{R}^n$. Then, for almost every $t\in(0,r_{\Omega_0} )$, we have
	\begin{equation*}
		-\dfrac{d}{dt}P(\Omega_{0,t})\geq n(n-1)W_2(\Omega_{0,t})
	\end{equation*}
	and equality holds if $\Omega_0$ is a ball.
\end{lem}
By simply applying the chain rule formula and recalling that $|D d_e(x)|=1$ almost everywhere, it remains proved the following.
\begin{lem}
	Let $f:[0,+\infty)\to[0,+\infty) $  be a non decreasing $C^1$ function   and let $\tilde f:[0,+\infty)\to[0,+\infty) $  a non increasing  $C^1$ function. We define  $u(x):=f(d_e(x))$,  $\tilde u(x):=\tilde f(d_e(x))$ and 
	\begin{equation*}
	\begin{split}
	E_{0,t}:=\{x\in \Omega_0\;:\; u(x)>t\},\\
	\tilde E_{0,t}:=\{x\in \Omega_0\;:\; \tilde u(x)<t\}.\\
	\end{split}
	\end{equation*}
	Then,
	\begin{equation}\label{derivata_composta_super1}
		-\dfrac{d}{dt}P(E_{0,t})\geq n (n-1)\dfrac{W_2(E_{0,t})}{|D u|_{u=t}},
	\end{equation}
	and 
		\begin{equation}\label{derivata_composta_sub1}
		\dfrac{d}{dt}P(\tilde E_{0,t})\geq n (n-1)\dfrac{W_2(\tilde E_{0,t})}{|D \tilde u|_{\tilde u=t}}.
	\end{equation}
\end{lem}

\section{Proof of the main result}
In this section we state and prove the main results. In the first theorem, we study the problem \eqref{eigRN}, in the second one the problem \eqref{torRN}. We consider a set $\Omega$ as defined at the beginning of Section $2$.
\begin{thm}\label{main_1}
	Let $\beta\in\mathbb{R}\setminus\{0\}$ and let $\Omega$ be such that $\Omega =\Omega_0\setminus\overline \Theta$, where $\Omega_0\subseteq\mathbb{R}^n$ is an open bounded and convex set  and $\Theta\subset\subset \Omega_0$ is a finite union of sets, each of one  homeomorphic to a ball of $\mathbb{R}^n$ and with Lipschitz boundary.  Let     $A=A_{r_1,r_2}$  be the annulus having the same measure of $\Omega$ and such that $P(B_{r_2})=P(\Omega_0)$. Then,
\begin{gather*}
	\label{in_eig_RN}		\lambda_p^{RN}(\beta, \Omega)\leq \lambda_p^{RN}(\beta, A).
							\end{gather*}
\end{thm}
\begin{proof} We divide the proof in two cases, distinguishing the sign of the Robin boundary parameter.\\
{\bf Case 1: $\beta > 0$.} We start by considering problem \eqref{eigRN} 
with positive value of the Robin  parameter. 
The  solution $v$ to   \eqref{eigRN}  is a radial function by Proposition \ref{radial_theorem} and  we denote by $v_m$ and $v_M$ the minimum and the maximum of $v$ on $A$. 
We construct the following test function defined in $\Omega_0$: 
\begin{equation}\label{utest}
u(x):=
\begin{cases}
G(d_e(x))\quad\ \ \text{if} \  d_e(x)< r_2-r_1\\
v_M\qquad\qquad\text{if} \ d_e(x)\geq r_2-r_1,
\end{cases}
\end{equation}
where $G$ is defined as 
\begin{equation*}
	G^{-1}(t)=\int_{v_m}^{t}\dfrac{1}{g(\tau)}\;d\tau,
\end{equation*}
 with  $g(t)=|Dv|_{v=t}$, defined for $v_m\leq t<v_M$, and $d_e(\cdot)$ denotes the distance from $\de\Omega_0$. We observe that $v(x)=G(r_2-|x|)$ and $u$ satisfy  the following properties: $u\in W^{1,p}(\Omega_0)$ and 
\begin{gather*}
|Du|_{u=t}=|Dv|_{v=t},\\
u_m:=\min_{\Omega_0} u=v_m=G(0),\\
u_M:=\max_{\Omega_0} u\leq v_M.
\end{gather*}
We need now to define the following sets:
\begin{equation}\label{superlevelsets}
\begin{split}
	 E_{0,t}:= & \{ x\in\Omega_0\;:\; u(x)>t  \}, \\
	 A_{t}:= & \{x\in A\;:\; v(x)>t\},\\
	 	 A_{0,t}:= & A_t\cup \overline{B}_{r_1}.
\end{split}
\end{equation}
For simplicity of notation, we will denote by $A_0$ the set $A_{0,0}$, i.e. the  ball $B_{r_2}$.
	Since $E_{0,t}$ and $A_{0,t}$ are  convex sets,   inequalities \eqref{derivata_composta_super1} and \eqref{iso_W_2} imply
	\begin{equation*}
-\dfrac{d}{dt}P(E_{0,t})\geq n(n-1)\dfrac{W_2(E_{0,t})}{g(t)}\geq n(n-1)n^{-\frac{n-2}{n-1}}\omega_n^{\frac{1}{n-1}}\dfrac{\left(P(E_{0,t})\right)^{\frac{n-2}{n-1}}}{g(t)},
	\end{equation*}
 for $u_m<t<u_M$. Moreover, it holds
	\begin{equation*}
	-\dfrac{d}{dt} P(A_{0,t})=n(n-1)n^{-\frac{n-2}{n-1}} \omega_n^{\frac{1}{n-1}} \dfrac{\left(P(A_{0,t})\right)^{\frac{n-2}{n-1}}}{g(t)},
	\end{equation*}
for $v_m<t<v_M$. Since, by hypothesis, $P(\Omega_0)=P(B_{r_2})$, using  a comparison type theorem, we obtain 
	\begin{equation*}
P(	E_{0,t})\leq P(A_{0,t}), 
	\end{equation*} 
for $v_m\leq t<u_M$. Let us also observe that 
\begin{equation}\label{perimetro_vivo}
\mathcal{H}^{n-1}(\partial E_{0,t}\cap\Omega)\leq P(E_{0,t})\leq P(A_{0,t}).
\end{equation}
Using now the coarea formula and \eqref{perimetro_vivo}:
\begin{multline}\label{gradient_estimates_e+}
\int_{\Omega}|Du|^p\;dx=
\int_{u_m}^{u_M} g(t)^{p-1}\;\mathcal{H}^{n-1}\left(\partial E_{0,t}\cap\Omega \right)dt \\ 
 \leq \int_{u_m}^{u_M} g(t)^{p-1}P(E_{0,t})\;dt\leq \int_{v_m}^{v_M} g(t)^{p-1}P(A_{0,t})\;dt=\int_{A}|Dv|^p\:dx.
\end{multline}
Since, by construction, $u(x)=u_m=v_m$ on $\Gamma_0$, then 
\begin{equation}\label{termine_bordo_me}
\int_{\Gamma_0}u^p\:d\mathcal{H}^{n-1}=u^p_m P(\Omega_0)=v^p_m P(A_0)=\int_{\de A_0}v^p\;d\mathcal{H}^{n-1}.
\end{equation}
Now, we define $\mu(t)=|E_{0,t}\cap\Omega|$ 
and $\eta(t)=|A_{t}|$ and using again coarea formula, we obtain, for $v_m\leq t<u_M$,
\begin{multline*}
\mu'(t)=-\int_{\{ u=t\}\cap\Omega}\dfrac{1}{|Du(x)|}\;d\mathcal{H}^{n-1}=-\dfrac{\mathcal{H}^{n-1}\left(  \partial E_{0,t}\cap\Omega\right)}{g(t)}\geq -\dfrac{P(E_{0,t})}{g(t)}\\\geq -\dfrac{P(A_{0,t})}{g(t)}=-\int_{\{ v=t\}}\dfrac{1}{|Dv(x)|}\;d\mathcal{H}^{n-1}=\eta'(t).
\end{multline*}
This inequality holds true also if $0<t<v_M$. 
Since $\mu(0)=\eta(0)$ (indeed $|\Omega|=|A|$), by integrating from $0$ to $t$, we have:
\begin{equation}\label{magmu}
\mu(t)\geq\eta(t),
\end{equation} for $0\leq t<v_M$. 
If we  consider the eigenvalue problem \eqref{eigRN}, we have 
\begin{equation}\label{L_p_estimates_e+}
	\int_{\Omega}u^p\;dx= \int_{v_m}^{v_M}pt^{p-1}\mu(t)dt\geq\int_{v_m}^{v_M}pt^{p-1}\eta(t)\;dt= \int_{A}v^p \;dx.
\end{equation}
Using \eqref{gradient_estimates_e+}-\eqref{termine_bordo_me}-\eqref{L_p_estimates_e+}, we achieve
\begin{equation*}
\begin{split}
\lambda_p^{RN}(\beta, \Omega)&\leq \dfrac{\ds\int_{\Omega}|Du|^p\;dx+\beta\ds\int_{\Gamma_0}u^p\;d\mathcal{H}^{n-1}    }{\ds\int_{\Omega}u^p\;dx} \\ 
& \leq \dfrac{\ds\int_{A}|Dv|^p\;dx+\beta\ds\int_{\de A_0}v^p\;d\mathcal{H}^{n-1}    }{\ds\int_{A}v^p\;dx}=\lambda^{RN}_p(\beta, A).
\end{split}
\end{equation*}

{\bf Case 2: $\beta <0$.} We consider now the problem \eqref{eigRN} 
with negative Robin external boundary parameter. By Proposition \ref{sign} the first $p$-Laplacian eigenvalue is negative. 
We observe that $v$ is a radial function. We construct now the following test function defined in $\Omega_0$:
\begin{equation*}
u(x):=
\begin{cases}
G(d_e(x))\quad\ \ \text{if} \  d_e(x)< r_2-r_1\\
v_m\qquad\qquad\ \text{if} \ d_e(x)\geq r_2-r_1,
\end{cases}
\end{equation*}
where 
	$G$ is defined as
	\begin{equation*}
		G^{-1}(t)=\int_{t}^{v_M}\dfrac{1}{g(\tau)}\;d\tau,
	\end{equation*}
	with  $g(t)=|Dv|_{v=t}$, defined for $v_m<t\leq v_M$ with $v_m:=\min_A v$  and $v_M:=\max_A v$. We observe that $u$ satisfies the following properties: $u\in W^{1,p}(\Omega_0)$ and	\begin{gather*}
	|Du|_{u=t}=|Dv|_{v=t},\\
	u_m:=\min_\Omega u\geq v_m,\\
	u_M:=\max_\Omega u = v_M=G(0).
	\end{gather*}
	We need now to define the following sets:
	\begin{equation*}\begin{split}
	\tilde E_{0,t}= & \{x\in\Omega_0\;:\; u(x)<t     \},\\
		\tilde A_t= & \{x\in A\;:\; v(x)<t\};\\
	\tilde A_{0,t}= & \tilde A_t\cup \overline{B}_{r_1}.
	\end{split}
	\end{equation*}
 For simplicity of notation, we will denote by $\tilde A_0$ the set $\tilde A_{0,0}$, i.e. the  ball $B_{r_2}$.
	Since $\tilde E_{0,t}$ and $\tilde A_{0,t}$ are now convex sets, by inequalities \eqref{derivata_composta_sub1} and \eqref{iso_W_2}, we obtain 
	\begin{equation*}
	\dfrac{d}{dt}P(\tilde E_{0,t})\geq n (n-1)\dfrac{W_2(\tilde E_{0,t})}{g(t)}\geq n(n-1)n^{-\frac{n-2}{n-1}}\omega_n^{\frac{1}{n-1}}\dfrac{\left(P(\tilde E_{0,t})\right)^{\frac{n-2}{n-1}}}{g(t)}.
	\end{equation*}
	Moreover, it holds
	\begin{equation*}
	\dfrac{d}{dt} P(\tilde A_{0,t})=n(n-1)n^{-\frac{n-2}{n-1}} \omega_n^{\frac{1}{n-1}} \dfrac{\left(P(\tilde A_{0,t})\right)^{\frac{n-2}{n-1}}}{g(t)}.
	\end{equation*}
	Since, by hypothesis, $P(\Omega_0)=P(B_{r_2})$, using  a comparison type theorem, we obtain 
	\begin{equation*}
	P(\tilde E_{0,t})\leq P(\tilde A_{0,t}), 
	\end{equation*} 
	for $u_m\leq t<v_M$. Moreover, we have 
	\begin{equation}\label{perimetro_vivo_sub1}
	\mathcal{H}^{n-1}(\partial \tilde E_{0,t}\cap\Omega)\leq P(\tilde E_{0,t})\leq P(\tilde A_{0,t}).
	\end{equation}
Using the coarea formula and \eqref{perimetro_vivo_sub1}, 
	\begin{equation}
	\label{gradient_estimates_e-}
	\begin{split}
\int_{\Omega} |Du|^p & \;dx= 
\int_{u_m}^{u_M} g(t)^{p-1}\;\mathcal{H}^{n-1}(\partial \tilde E_{0,t}\cap\Omega)\;dt  \\ 
& \leq \int_{u_m}^{u_M} g(t)^{p-1}P(\tilde E_{0,t})\;dt\leq \int_{v_m}^{v_M} g(t)^{p-1}P(\tilde A_{0,t})\;dt=\int_{A}|Dv|^p\:dx.
	\end{split}
	\end{equation}	
	Since, by construction, $u(x)=u_M=v_M$ on $\Gamma_0$, it holds
\begin{equation}\label{termine_bordo_Me}
\int_{\Gamma_0}u^p\:d\mathcal{H}^{n-1}=u^p_M P(\Omega_0)= v^p_M P(A_0)=\int_{\de A_0}v^p\;d\mathcal{H}^{n-1}.
\end{equation}
	We define now  $\tilde{\mu}(t)=|\tilde E_{0,t}\cap \Omega|$ 
	and $\tilde{\eta}(t)=|\tilde A_{t}|$ and
	using coarea formula, we obtain, for $u_m\leq t<v_M$,
	\begin{multline*}
	\tilde\mu'(t)=\int_{\{ u=t\}\cap\Omega}\dfrac{1}{|Du(x)|}\;d\mathcal{H}^{n-1}=\dfrac{\mathcal{H}^{n-1}(\partial \tilde E_{0,t}\cap\Omega)}{g(t)}\leq \dfrac{P(\tilde E_{0,t})}{g(t)}\\\leq \dfrac{P(\tilde A_{0,t})}{g(t)}=\int_{\{ v=t\}}\dfrac{1}{|Dv(x)|}\;d\mathcal{H}^{n-1}=\tilde\eta'(t).
	\end{multline*}
Hence $\mu'(t)\leq\eta'(t)$ for $v_m\leq t \leq v_M$. Then, by integrating from $t$ and $v_M$: 
\begin{equation*}
|\Omega|-\tilde\mu(t)\leq |A|-\tilde\eta(t),
\end{equation*}
 for $v_m\leq t<v_M$ and consequently $\tilde\mu(t)\geq \tilde\eta(t)$. 

	Let us consider the eigenvalue problem \eqref{eigRN}. We have that 
	
	\begin{equation}	\label{L_p_estimates_neg_e-}
	\int_{\Omega}u^p\;dx = u_M^p|\Omega|-\int_{u_m}^{u_M}pt^{p-1}\tilde\mu(t)dt \leq    v_M^p|A|-\int_{v_m}^{v_M}pt^{p-1}\tilde\eta(t)\;dt= \int_{A}v^p\;dx.
	\end{equation}
	
	By \eqref{gradient_estimates_e-}-\eqref{termine_bordo_Me}-\eqref{L_p_estimates_neg_e-}, we have
	\begin{multline*}\label{thesis1}
	\lambda_p^{RN}(\beta, \Omega)\leq \dfrac{\ds\int_{\Omega}|Du|^p\;dx+\beta\ds\int_{\Gamma_0}u^p\;d\mathcal{H}^{n-1}    }{\ds\int_{\Omega}u^p\;dx}\leq \\\leq \dfrac{\ds\int_{A}|Dv|^p\;dx+\beta\ds\int_{\de A_0}v^p\;d\mathcal{H}^{n-1}    }{\ds\int_{A}v^p\;dx}=\lambda^{RN}_p(\beta, A).
	\end{multline*}	
		
	\end{proof}
	\begin{thm}\label{main_2}
	Let $\beta>0$ and let $\Omega$ be such that $\Omega =\Omega_0\setminus\overline \Theta$, where $\Omega_0\subseteq\mathbb{R}^n$ is an open bounded and convex set  and $\Theta\subset\subset \Omega_0$ is a finite union of sets, each of one  homeomorphic to a ball of $\mathbb{R}^n$ and with Lipschitz boundary.  Let     $A=A_{r_1,r_2}$  be the annulus having the same measure of $\Omega$ and such that $P(B_{r_2})=P(\Omega_0)$. Then,
\begin{equation*}
			 \label{in_tor_RN}			T_p^{RN}(\beta, \Omega)\geq T_p^{RN}(\beta, A).
			 \end{equation*}
\end{thm}
\begin{proof} 
Let $v$ be the function that achieves the minimum in \eqref{torRN} on the annulus $A$. We consider the test function as in \eqref{utest} and the superlevel sets as in \eqref{superlevelsets}. By \eqref{magmu} we have
\begin{equation} \label{stima_pos_e+}
\int_{\Omega}u\;dx=\int_{0}^{v_M}\mu(t)dt\geq\int_{0}^{v_M}\eta(t)\;dt= \int_{A}v\;dx.
\end{equation}
In this way, using \eqref{gradient_estimates_e+}-\eqref{termine_bordo_me}-\eqref{stima_pos_e+}, we conclude
\begin{equation*}
\begin{split}
\frac{1}{T_p^{RN}(\beta, \Omega)} & \leq \dfrac{\ds\int _{\Omega}|Du|^p\;dx+\beta\ds\int_{\Gamma_0}u^p \;d\mathcal{H}^{n-1}}{\left|\ds\int_{\Omega}u\;dx\right|^p}\\ 
&\leq\dfrac{\ds\int _{A}|Dv|^p\;dx+\beta\ds\int_{\partial A_0}v^p \;d\mathcal{H}^{n-1}}{\left|\ds\int_{A}v\;dx\right|^p}=\dfrac{1}{T_p^{RN}(\beta, A)}.
\end{split}
\end{equation*}
\end{proof}

	We conclude with some remarks.
\begin{rem}
	In \cite{AA} the authors prove that the annulus maximizes the first eigenvalue of the $p$-Laplacian with Neumann condition on internal boundary and Dirichlet  condition on  external boundary, among sets of $\mathbb{R}^n$ with holes and having a sphere as outer boundary. 
	We explicitly observe that our result includes this case, since 
	\begin{equation*}
	\lim\limits_{\beta\to+\infty} \lambda^{RN}_p(\beta,\Omega)=\lambda^{DN}_p(\Omega),
	\end{equation*}
	where with $\lambda^{DN}_p(\Omega)$ we denote the first eigenvalue of the $p$-Laplacian endowed with  Dirichlet condition on external boundary  and Neumann  condition on  internal boundary.
	
\end{rem}

\begin{rem}
	Let us remark that in the case $p=2$, we know esplicitely the expression of the solution of the problems described in the paper on the annulus $A=A_{r_1,r_2}$.
	
We denote by  $J_{\nu }$ and $Y_{\nu}$, respectively, the Bessel functions of the first and second kind of order $\nu$. The function that achieves the minimum in $\lambda=\lambda_p^{RN}(\beta, A)$ is
	\begin{equation*}
	v(r)=Y_{\frac n2-2}(\sqrt\lambda r_2)r^{1-\frac n2}J_{\frac n2-1}(\sqrt\lambda r)-J_{\frac n2 - 2}(\sqrt\lambda r_2)r^{1-\frac n 2}Y_{\frac n2 -1}(\sqrt\lambda r),
	\end{equation*}
	with the condition
	\begin{equation*}
	\begin{split}
	Y_{\frac n2-2}(\sqrt\lambda r_1)[r_2^{1-\frac n2}J_{\frac n2-2}(\sqrt\lambda r_2)\sqrt\lambda+\beta r_2^{1-\frac n 2}J_{\frac n2-1}(\sqrt\lambda r_2)]-\qquad\\
	J_{\frac n2-2}(\sqrt\lambda r_1)[r_2^{1-\frac n 2}Y_{\frac n2 -2}(\sqrt\lambda r_2)\sqrt\lambda + \beta r_2^{1-\frac n2}Y_{\frac n2 - 1}(\sqrt\lambda r_2)]=0.
	\end{split}
	\end{equation*}

	The function that achieves the minimum  $1/T=1/T_p^{RN}(\beta, A)$ is
	\begin{equation*}
	v(r)=\frac{1}{2Tn}r^2+c_1\frac{(1-n)}{r^n}+c_2,
	\end{equation*}
	with
	\begin{equation*}
	\begin{sistema}
	c_1=\frac{1}{\beta T}\left(\frac{r_2}{n}-\frac{r_1^n}{nr_2^{n-1}}+\frac{\beta r_2^2}{2n}+\frac{(n-1)\beta}{n}\left(\frac{r_1}{r_2}\right)^n\right)\\
	c_2=-\frac{1}{nT}r_1^n.
	\end{sistema}
	\end{equation*}
	
\end{rem}  

\section*{Acnowledgements}
This work has been partially supported by GNAMPA of INdAM. The second author (G. Pi.) was also supported by Progetto di eccellenza \lq\lq Sistemi distribuiti intelligenti\rq\rq of Dipartimento di Ingegneria Elettrica e dell'Informazione \lq\lq M. Scarano\rq\rq. 

Moreover, we would like  to thank the reviewer for his/her suggestions 
 to improve this paper.
\small{

}

\end{document}